\documentclass[11pt]{amsart}
\usepackage{amsmath,amssymb,amsthm}
\usepackage{mathtools}
\usepackage{stmaryrd} 
\usepackage{graphicx,setspace}
\usepackage{tikz}
\usetikzlibrary{arrows}
\usepackage{wasysym}
\usepackage[makeroom]{cancel}
\usepackage{fullpage}
\usepackage{comment}
\usepackage{enumerate}
\usepackage{mathtools}
\usepackage{url}
\usepackage{tipa}
\usepackage{upgreek}
\usepackage{csquotes}
\allowdisplaybreaks

\newcommand{\Zmod}[1]{\mathbb{Z}/#1\mathbb{Z}}
\newcommand{\Emod}[1]{E\left(\Zmod{#1}\right)}
\newcommand{\ENZ}{\Emod{N}}
\newcommand{\zero}{\mathcal{O}}

\newcommand{\characteristic}{\text{char}}

\newcommand{\jacsym}[2]{\left( \frac{#1}{#2} \right)}

\theoremstyle{plain}
\newtheorem{theorem}{Theorem}[section]
\newtheorem{proposition}[theorem]{Proposition}

\newtheorem{lemma}[theorem]{Lemma}
\newtheorem*{notation*}{Notation}
\newtheorem{corollary}[theorem]{Corollary}

\newtheorem*{claim*}{Claim}
\theoremstyle{definition}

\newtheorem{definition}[theorem]{Definition}

\providecommand{\abs}[1]{\left\lvert#1\right\rvert}

\def\N{\mathbb{N}}
\def\Z{\mathbb{Z}}
\def\Q{\mathbb{Q}}

\def\S{\mathbb{S}}
\def\mcO{\mathcal{O}}

\def\ord{\operatorname{ord}}

\renewcommand\ord\nu
\usepackage{hyperref}
\DeclareMathOperator{\order}{order}

\DeclareMathOperator{\lcm}{lcm}

\begin{document}

\title{Strongly Non-zero Points and Elliptic Pseudoprimes}

\author[L. Babinkostova]{L. Babinkostova}
\address{Department of Mathematics\\ Boise State University\\
Boise, ID 83725, USA}
\email{liljanababinkostova@boisestate.edu}

\author[D. Fillmore]{D. Fillmore}
\address{Department of Mathematics\\ University of South Carolina\\
Columbia, SC 29208, USA}
\email{Dylanf@email.sc.edu}

\author[P. Lamkin]{P. Lamkin}
\address{Department of Mathematics\\ Carnegie Mellon University\\
Pittsburgh, PA 15213, USA}
\email{plamkin@andrew.cmu.edu}

\author[A. Lin]{A. Lin}
\address{Department of Mathematics\\ Princeton University\\
Princeton, NJ 08544, USA}
\email{adlin@princeton.edu }

\author[C. L. Yost-Wolff]{C. L. Yost-Wolff}
\address{Department of Mathematics\\ Massachusetts Institute of Technology\\
Cambridge, MA 02139, USA}
\email{calvinyw@mit.edu}

\thanks{Supported by the National Science Foundation under the Grant number DMS-1659872.}
\thanks{$^{\S}$ Corresponding Author: liljanababinkostova@boisestate.edu}

\begin{abstract}
We examine the notion of strongly non-zero points and use it as a tool in the study of several types of  elliptic pseudoprimes introduced in \cite{Gordon-psp}, \cite{Silverman} and \cite{REU2017}. Moreover, we give give some probabilistic results about the existence of strong elliptic pseudoprimes for a randomly chosen point on a randomly chosen elliptic curve.
\end{abstract}
\maketitle

\section{Introduction}

The notion of testing a number for primality has long been an interesting problem in mathematics.  Possibly the most well-known primality test is based on Fermat's Little Theorem: if $p$ is a prime number and $b$ is an integer not divisible by $p$, then $b^{p-1} \equiv 1 \pmod{p}$. However, the converse does not hold: there are composite numbers $N$ and positive integers $1< b<N$ for which $b^{N-1} \equiv 1 \pmod{N}$. We refer the reader to the survey article by C. Pomerance \cite {P} for a nice introduction to primality testing. Clasically, a natural number $N$ is a \textit{pseudoprime} to the base $b$ if $N$ is composite and $b^{N-1}\equiv 1 \bmod{N}$. If $N$ is a pseudoprime for all $b$ with $\gcd(b,N)=1$ then $N$ is called \textit{Carmichael number}.   In \cite{KC}, Korselt characterized these numbers as follows: $N$ is a Carmichael number if and only if $N$ is square-free and $p-1 \mid N - 1$ for every prime $p \mid N$. In 1986, the long-standing conjecture that there are infinitely many Carmichael numbers was proven by Alford, Granville, and Pomerance \cite {AGP}.\\

Since the 1980's, elliptic curves have been used in algorithmic number theory to give deterministic algorithms that are faster than earlier algorithms that did not use elliptic curves. We refer the reader to \cite{Lenstra-factoring} for historical remarks on elliptic curve primality testing. The general framework of elliptic curve primality testing is based on the following fundamental theorem of Goldwasser and Kilian \cite{GK}.
\begin{theorem}{\cite{GK}}
Let $E/\mathbb{Q}$ be an elliptic curve, and let $M$ and $N$ be positive integers with $M>(N^{1/4 + 1})2$ and $N$ is coprime to $\Delta(E)$. Suppose there is a point $P\in E/\mathbb{Q}$ such that $MP$ is zero$\mod N$ and $(M/p)P$ is strongly non-zero$\mod N$ for every prime $p\mid M$. Then $N$ is prime. 
\end{theorem}
Although the original algorithm of Goldwasser-Kilian is no longer used, their result is used as a framework for the ``AKS" primality test,  developed by Agrawal, Kayal, and Saxena in \cite{AKS}, which is the only known algorithm that determines the primality or compositeness of any integer in deterministic polynomial time. \\

In 1992, Gordon introduced the notion of an {\em elliptic pseudoprime} \cite{Gordon-psp} as a natural extension of the definition of a pseudoprime from groups arising from elliptic curves with complex multiplication.
\begin{definition}\cite{Gordon-psp}
Let $E/\Q$ be an elliptic curve with complex multiplication by an order in $\Q(\sqrt{-d})$ and let $P \in E(\mathbb{Q})$ have infinite order. A composite number $N$ is called an \emph{elliptic pseudoprime} if $\jacsym{-d}{N} = -1$, $N$ is coprime to  $\Delta(E)$, and $N$ satisfies $(N+1)P\equiv \zero \pmod{N}$.
\end{definition}
We will use the notation ``G-pseudoprime" to denote Gordon's notion of an elliptic pseudoprime. In \cite{Silverman-Carmichael}, Silverman extends Gordon's  notion of elliptic pseudoprimes by allowing any elliptic curve $E/\Q$, not just elliptic curves with complex multiplication, as well as any $P \in \Emod{N}$. 
\begin{definition}\cite{Silverman-Carmichael}
Let $N \in \mathbb{Z}$, let $E/\mathbb{Q}$ be an elliptic curve, and let $P \in \Emod{N}$. Write the $L$-series of $E/\mathbb{Q}$ as $L(E/\mathbb{Q},s) = \sum_n \frac{a_n}{n^s}$. Then $N$ is an \emph{elliptic pseudoprime} for $(E,P)$ if $N$ has at least two distinct prime factors, $E$ has good reduction at every prime $p$ dividing $N$, and $(N+1-a_N) P \equiv \zero \pmod{N}$.
\end{definition}
{\flushleft {We will use the notation ``S-pseudoprime" to denote Silverman's notion of an elliptic pseudoprime.}} \\

In this paper we study elliptic G- and S- pseudoprimes for strongly non-zero points on the elliptic curve $\Emod{N}$ (Section 3).  Moreover, we give bounds on the number of points on a given elliptic curve for which an odd integer $N$  is a strong elliptic G-pseudoprime and probabilistic results for a given odd integer $N$ being a strong elliptic G- pseudoprime  for a randomly chosen point on a randomly chosen elliptic curve (Section 4). We prove similar results for strong elliptic S-pseudoprimes. Prior to these results we give a brief introduction to elliptic curves and elliptic pseudoprimes (Section 2)
\section{Preliminaries} 
\subsection{Elliptic Curves}
We introduce some elementary features of elliptic curves which are relevant to the topics presented in this paper. We refer the reader to \cite{Silverman} and \cite{Washington} for detailed introduction to elliptic curves. Let $k$ be a field and let $\overline{k}$ denote its algebraic closure. An \emph{elliptic curve} $E$ over a field $k$ is a non-singular {\footnote{\tiny an algebraic curve is said to be non-singular if there is not point on the curve at which all partial derivatives vanish.}} curve with an affine equation of the form

\begin{align} 
	E/k: y^2 + a_1xy + a_3y = x^3 + a_2x^2 + a_4x + a_6     \label{eq1}
\end{align}
where $a_1,a_2,a_3,a_4,a_6 \in k$. An equation of the above form (\ref{eq1}) is called a \emph{generalized Weierstrass} equation.

Recall that the points in projective space $\mathbb{P}^2(k)$ correspond to the equivalence classes in $k^3 - \{(0,0,0)\}$ under the equivalence relation $(x, y, z) \sim (ux, uy, uz)$ with $u\in k^{\times}$. The equivalence class containing $(x,y,z)$ is denoted by $[x:y:z]$. The projective equation corresponding to the affine equation (\ref{eq1}) is the homogeneous equation
\begin{align}
	E/k: y^2z + a_1xyz + a_3yz^2 = x^3 + a_2x^2z + a_4xz^2 + a_6z^3,  \label{eq}
\end{align}
where $a_1,a_2,a_3,a_4,a_6 \in \overline{k}$.  

The point $[0 : 1 : 0]$ is called the \emph{point at infinity} and is denoted by $\mathcal{O}$. The projective points of $E$ over $\overline{k}$ form an abelian group with $\mathcal{O}$ as the identity.

If $\characteristic(k)\neq 2,3 $, then the equation of $E$ can be written as
\begin{align*}
	E/k: y^2 = x^3 + Ax + B  \label{eq2}
\end{align*} 
where $A, B \in k$.

An elliptic curve $E/k: y^2z = x^3+Axz^2+Bz^3$ is non-singular if and only if its discriminant, $4A^3+27B^2$, is nonzero.  Associated to $E/\mathbb{Q}$ is the $L$-function $L(E,s)$, which is defined as the Euler product
	\begin{align*}
		L(E,s) = \prod_p \frac{1}{1- a_p p^{-s} + 1_E(p) p^{1-2s}}
	\end{align*}
	where
	\begin{align*}
	1_E(p) = \begin{cases} 1 &\text{if } E \text{ has good reduction at } p \\
													0 &\text{otherwise} \end{cases}
\end{align*}
 and $a_p = p+1-\#\Emod{p}$ whether or not $E$ has good reduction at $p$. Alternatively expressing $L(E,s)$ as the Dirichlet series $L(E,s) = \sum_n \frac{a_n}{n^s}$, the map sending a positive integer $n$ to the coefficient $a_n$ is a multiplicative function with
\begin{align*}
	a_{1} &= 1 \\
	a_{p^e} &= a_p a_{p^{e-1}} - 1_E(p) p a_{p^{e-2}} \qquad \text{for all } e \geq 2.
\end{align*}
See \cite[Chapter 8.3]{DiamondShurman} and \cite[Appendix C, Section 16]{Silverman} for more on $L$-series of elliptic curves. \par

An elliptic curve $E/\mathbb{Z}/N\mathbb{Z}$ is the set of solutions $[x:y:z]$ (requiring that $\gcd(x, y, z, N)=1$) in projective space over $\mathbb{Z}/N\mathbb{Z}$ to a Weierstrass equation
\begin{align*} 
	E/k: y^2 + a_1xy + a_3y = x^3 + a_2x^2 + a_4x + a_6    
\end{align*}
where the discriminant $\Delta$ has no prime factor in common with $N$. There is a group law on $\Emod{N}$ given by explicit formulae which can be computed (see \cite{Washington}). For a given elliptic curve $E/\mathbb{Q} : y^2=x^3 +Ax+B$ where $A, B,N \in \mathbb{Z}$ with $N$ positive odd integer such that $\gcd(N, 4A^3+ 27B^2)=1$ there is a group homomorphism from $E/\mathbb{Q}$ to $\Emod{N}$ by representing the points in $E/\mathbb{Q}$ as triples $[x:y:z]\in \mathbb{P}^2(k)$.

If the prime factorization of $N$ is $N = p_1^{e_1} \cdots p_k^{e_k}$ then $\Emod{N}$ is isomorphic as a group to the direct product of elliptic curve groups 
$$ \Emod{N} \simeq \Emod{p_1^{e_1}} \oplus \cdots \oplus \Emod{p_k^{e_k}}.$$
In particular, if we let $E_i$ be the reduction of $E$ modulo $p_i$, then $E_i$ is an elliptic curve over the field $\mathbb{F}_{p_i}$. It is known that
$$\#E(\mathbb{Z}/p_i^{e_i}\mathbb{Z})=p_i^{{e_i}-1}\#E_i(\mathbb{F}_{p_i})$$
We refer the reader to \cite{{Lenstra-factoring}, {Washington}} for details about elliptic curves over $\mathbb{Z}/N\mathbb{Z}$.

\subsection{Elliptic Pseudoprimes}
n this section we give some background on elliptic pseudoprimes in general. For other articles that study elliptic pseudoprimes and related notions see \cite{{Gordon-psp}, {Gordon-number}, {EPT}, {Ekstrom}, {Muller}, {Silverman-Carmichael}}. 
\begin{definition}\cite{Gordon-psp}\label{GordonEP}
Let $E/\Q$ be an elliptic curve with complex multiplication in $\Q(\sqrt{-d})$, let $P$ be a point in $E$ of infinite order, and let $N$ be a composite number with $\gcd(N,6\Delta) = 1$. Then, $N$ is an \emph{elliptic pseudoprime} for $(E,P)$ if $\jacsym{-d}{N} =-1$ and 
$$(N+1) P \equiv \mathcal{O}\pmod N\footnote{\tiny For details on computing multiples of points in elliptic curve modulo $N$, see \cite[Chapter 3.2]{W1} or Appendix \ref{SectionMult}.}$$
\end{definition}

In \cite{Silverman-Carmichael}, Silverman extends Gordon's aforementioned notion of elliptic pseudoprimes by allowing any elliptic curve $E/\Q$, not just elliptic curves with complex multiplication, as well as any $P \in \Emod{N}$. 
\begin{definition}\cite{Silverman-Carmichael}\label{SilvermanEP}
Let $N \in \mathbb{Z}$, let $E/\mathbb{Q}$ be an elliptic curve, and let $P \in \Emod{N}$. Write the $L$-series of $E/\mathbb{Q}$ as $L(E/\mathbb{Q},s) = \sum_n \frac{a_n}{n^s}$. Then $N$ is an \emph{elliptic pseudoprime} for $(E,P)$ if $N$ has at least two distinct prime factors, $E$ has good reduction at every prime $p$ dividing $N$, and $(N+1-a_N) P \equiv \zero \pmod{N}$.
\end{definition}
It is not hard to check that for (most) $N$, $\jacsym{-d}{N} = -1$ and $N$ is square-free if and only if $a_N = 0$. Thus, $(n + 1 - a_N)P = (n + 1)P$, so (most) elliptic pseudoprimes in Gordon's sense are also pseudoprimes in Silverman's sense.

\begin{definition}{\cite{Gordon-psp}}
Let $E/\Q$ be an elliptic curve. A composite number $N$ with $\gcd\left(N, 6\Delta\right) = 1$ is an elliptic \emph{G-pseudoprime} for the curve $E/\Q$ with complex multiplication by the field $K = \Q\left(\sqrt{-d}\right)$ and a point $P \in E\left(\Q\right)$ of infinite order if $\left(\frac{-d}{N} \right) = -1$ and
\[
	\left(N + 1\right)P \equiv \mcO \bmod{N}.
\]
\end{definition}

\begin{definition}{\cite{Gordon-psp}}
\label{def:strong G-psp}
Let $E/\Q$ be an elliptic curve with complex multiplication. Suppose $N$ is a composite number with $\gcd\left(N, 6\Delta\right) = 1$. Write $N+1 = 2^st$ where $t$ is odd. Then $N$ is called a  \emph{strong elliptic G-pseudoprime} for a curve $E$ with complex multiplication by $K = \Q\left(\sqrt{-d}\right)$ and a point $P \in E\left(\Q\right)$ with infinite order if $\left(\frac{-d}{N} \right) = -1$ and either
\begin{enumerate}
\renewcommand{\labelenumi}{(\roman{enumi})}
\item $tP \equiv \mcO \bmod{N}$, or
\item $\left(2^rt\right)P \equiv \left(x:0:1\right) \bmod{N}$ for some $0 \leq r \leq s-1$ and some $x \in \Z/N\Z$.
\end{enumerate}
\end{definition}

\begin{definition}{\cite{Gordon-psp}}
Let $E/\Q$ be an elliptic curve. A composite number $N$ is an elliptic \emph{(strong) G-Carmichael number for $E$} if it is a (strong) G-pseudoprime for $E$ at all points $P \in E\left(\Z/N\Z\right)$.
\end{definition}

\begin{definition}{\cite{Silverman-Carmichael}}
Let $E/\Q$ be an elliptic curve and it's associated L-series be $L(E,s) = \sum_{n \ge 1} a_n/n^s$. A composite number $N$ is an elliptic \emph{S-pseudoprime} for $E/\Q$ and a point $P \in E\left(\Z/N\Z\right)$ if $N$ has at least two distinct prime factors, $E$ has good reduction at every prime $p \mid N$, and
\[
	\left(N + 1 - a_N\right)P \equiv \mcO \bmod{N}.
\]
\end{definition}
In \cite{REU 2017}, the authors extend the notion of a strong elliptic G-pseudoprime by considering non-CM curves.
\begin{definition} {\cite{REU2017}}\label{def:strong S-psp}
Let $E/\Q$ be an elliptic curve and its associated L-series be $L(E,s) = \sum_{n \ge 1} a_n/n^s$. Let $N$ be an integer, and let $P$ be a point in $E(\Z/N\Z)$. Write $N+1-a_N = 2^st$, where $t$ is odd. Then, $N$ is a \emph{strong elliptic S-pseudoprime} for $(E,P)$ if $N$ has at least two distinct prime factors, $E/\Q$ has good reduction at every prime $p \mid N$, and one of the following holds:
\begin{enumerate}[(i)]
\item $tP \equiv \mcO \bmod{N}$, or
\item $\left(2^rt\right)P \equiv \left(x:0:1\right) \bmod{N}$ for some $0 \leq r \leq s-1$ and some $x \in \Z/N\Z$.
\end{enumerate}
\end{definition}
From these definitions of S-pseudoprimes for a specific point $P$ on a curve $E$, it is natural to extend the idea of Carmichael numbers for the group $(\Z/N\Z)^\times$ to Carmichael numbers for the group $E(\Z/N\Z)$. 
\begin{definition}[{\cite{Silverman-Carmichael},  \cite{REU2017}}]
Let $E/\Q$ be an elliptic curve. A composite number $N$ is a \emph{(strong) elliptic S-Carmichael number for $E/\Q$} if it is a (strong) elliptic S-pseudoprime for $E/\Q$ at all points $P \in E\left(\Z/N\Z\right)$.
\end{definition}

\section{Strongly Nonzero Points and Elliptic Pseudoprimes}
In this section we use the notion of strongly non-zero points and use it as a tool for examining G- and S- elliptic Carmichael numbers. 
\begin{definition}\label{def:stronglynonzero}
Let $P = (x : y : z)$ be a projective point on an elliptic curve $E/\Q$, where $x, y, z \in \Z$, and let $N$ be a nonzero integer. If $z = 0 \mod N$ then the point $P$ is said to be \textit{zero$\mod N$}; otherwise, $P$ is \textit{non-zero$\mod N$}. If $gcd(z, N) =1$ then the point $P$ is said to be \textit{strongly non-zero$\mod N$}. 
\end{definition}

Note that if $P$ is strongly non-zero$\mod N$, then $P$ is non-zero$\mod p$ for every prime $p\vert N$. When $N$ is prime, the notions of nonzero and strongly non-zero coincide. 

\begin{lemma}\label{lemma:order1}
 Let $Q$ be a strongly non-zero point on the elliptic curve $E\left(\Z/N\Z\right)$. 
 Consider the group decomposition
\[
	E\left(\mathbb{Z}/{N\mathbb{Z}}\right) \cong \bigoplus_{p \mid N} E\left(\mathbb{Z}/{p^{\nu_p\left(N\right)}\mathbb{Z}}\right).
\]
where $\nu_p\left(N\right)$ denotes the $p$-adic valuation of $N$. Let $Q_p\in E\left(\Z/p^{\nu_p\left(N\right)}\Z\right)$ denote the point corresponding to $Q$ for a prime $p \mid N$. Then $Q_p$ is a strongly non-zero point$\mod p^{\nu_p(N)}$ for all $p \mid N$.
\end{lemma}

\begin{proof}
Since $Q$ is strongly non-zero point, we may write $Q=\left(x:y:z \right)$ with $z=1$. 
Then $Q_p=\left(x \bmod p^{\ord_p\left(N\right)}: y \bmod p^{\ord_p\left(N\right)}: z \bmod p^{\ord_p\left(N\right)}\right)$. Note that for any integer $k>0$, 
\begin{align*}
  \gcd\left(p^{\ord_p\left(N\right)},z + kp^{\ord_p\left(N\right)}\right)= \gcd\left(p^{\ord_p\left(N\right)},z\right).
\end{align*}
Also, since $p^{\nu_p\left(N\right)} \mid N$,
\begin{align*}
  \gcd\left(p^{\ord_p\left(N\right)},z\right) \mid \gcd\left(N,z\right) = 1
\end{align*}
Thus $\gcd\left(p^{\ord_p\left(N\right)},z \bmod p^{\ord_p\left(N\right)}\right) =1$, which implies that $Q_p$ is strongly non-zero point$\mod p^{\nu_p(N)}$.
\end{proof}

\begin{corollary}\label{lemma:order2}
Let $Q$ be a point in $E(\Z/N\Z)$, and let $Q_p$ as defined above. Then $Q$ is a zero point$\mod N$ if and only if there exists a prime $p \mid N$ such that $Q_p$ is a zero point$\mod N$.
\end{corollary}
Throughout the rest of the section we consider the case when $E(\Z/{N\Z})$ has strongly non-zero points$\mod N$.

\begin{proposition}\label{lemma:order3}
Let $E(\Z/p^m\Z)$ be an elliptic curve and $Q\in E(\Z/p^m\Z)$ a point. Let $\sigma_{m,n}: E\left(\Z/p^m\Z\right) \rightarrow E\left(\Z/p^n\Z\right)$, $m \geq n$ be the homomorphism given by  $$\sigma_{m,n}\left(Q\right)=\left(x \bmod p^n:y \bmod p^n:z \bmod p^n\right)$$ Then $\sigma_{m,n}\left(Q\right)$ is a non-zero point  in  $E\left(\Z/p^n\Z\right)$ if and only if $Q$ is a non-zero point in  $E\left(\Z/p^m\Z\right)$.
\end{proposition}

\begin{proof}
Write $Q=\left(x:y:z\right)$. Then $\sigma_{m,n}\left(Q\right)=\left(x \bmod p^n:y \bmod p^n:z \bmod p^n\right)$. Then for any integer $k$, $p\mid \left(z - kp^n\right)$ if and only if $p\mid z$.  Since $p$ is prime, for any integer $i>0$, if $\gcd\left(p^i,z\right)>1$, then $p\mid z$. It follows that $\gcd\left(p^n,z \bmod p^n\right)>1$ if and only if $\gcd\left(p^m,z \bmod p^n\right)>1$.
\end{proof}

\begin{corollary}\label{lemma:order4}
 If $Q$ is a non-zero point on the elliptic curve $E\left(\Z/p^m\Z\right)$, then $\vert Q \vert=p^k$ for some integer $k < m$.
\end{corollary}

\begin{proof}
Let $\sigma_{m,1}: E\left(\Z/p^m\Z\right) \rightarrow E\left(\Z/p\Z\right)$ be the homomorphism as in Lemma \ref{lemma:order3}. Note that the only non-zero point in $E\left(\Z/p\Z\right)$ is the identity $\mcO$.
By Lemma \ref{lemma:order3}, $\ker\left(\sigma_{m,1}\right)$ is the set of all non-zero points in $E\left(\Z/p^m\Z\right)$. Also, from $E_{m-1}/E_m \cong \ker(\sigma_{m,m-1})$ and $\ker(\sigma_{f,f-1}) \cong \Z/p\Z$ we have that $\abs{\ker\left(\sigma_{m,1}\right)}=p^{m-1}$. This implies that that $\vert Q\vert \mid p^{m-1}$. Thus $\vert Q\vert=p^k$ for some integer $0 \le k < m$. 
\end{proof}
\begin{lemma}\label{lemma:order6}
Let $Q$ be a non-zero point on the elliptic curve $E\left(\Z/p^n\Z\right)$ and $k$ coprime to $p$. Then there exists a strongly non-zero point $P \in E\left(\Z/p^n\Z\right)$ such that $kP=Q$ if and only if there exists a strongly nonzero point $P' \in E\left(\Z/p^n\Z\right)$ with $\vert P'\vert $ dividing $k$.
\end{lemma}

\begin{proof}
Let $Q$, $E\left(\Z/p^n\Z\right)$ and $k$ be given. Let $P' \in E(\Z/p^n\Z)$ a strongly non-zero point such that $\vert P'\vert$ divides $k$. Since $\gcd(k,p^n) = 1$, there exists a positive integer $y$ such that $ky \equiv 1 \bmod p^n$. Let $P= yQ+P'$. Thus by Corollary \ref{lemma:order4}, $P$ is a strongly non-zero point. Note that
\[
  kP= kyQ+kP'= kyQ + \mcO = kyQ = Q
\] 
Conversely, assume that $P$ with $kP=Q$ is a a strongly nonzero point. Let $y$ be a positive integer such that $ky = 1 \bmod p^n$. Let $P' = P-yQ$.  By Corollary \ref{lemma:order4}, 
$P'$ is a strongly non-zero point. Note that
\[
  kP'= kP-kyQ= Q - Q = \mcO
\]
Therefore $\vert P' \vert |k$.
\end{proof}

\begin{lemma}\label{lemma:order7}          
Let $Q$ be a non-zero point in $E\left(\Z/p^n\Z\right)$ and let  $\sigma_{n,n-1}: E\left(\Z/p^n\Z\right) \rightarrow E\left(\Z/p^{n-1}\Z\right)$ be the natural homomorphism. Let $k$ be an integer and $P' \in E\left(\Z/p^{n-1}\Z\right)$ be a strongly non-zero point such that $kP'=\sigma_{n,n-1}\left(Q\right)$ and $p \nmid k$. Then there exists a point $P\in E\left(\Z/p^n\Z\right)$ such that $kP=Q$.
\end{lemma}

\begin{proof}
Note that $\vert\ker\left(\sigma_{n,n-1}\right)\vert=p$. It follows that $\ker\left(\sigma_{n,n-1}\right)\cong \mathbb{Z}/{p\mathbb{Z}}$. We can write 
\[
  E\left(\Z/p^n\Z\right) \cong \bigoplus_i \mathbb{Z}/{p^{a_i}\mathbb{Z}}  \oplus G
\]
where $G$ does not contain any elements of order $p$.  Since $\ker\left(\sigma_{n,n-1}\right)\cong \mathbb{Z}/{p\mathbb{Z}}$ is a normal subgroup of $E\left(\Z/p^n\Z\right)$, it follows that
\[
  E\left(\Z/p^{n-1}\Z\right) \cong \bigoplus_i \mathbb{Z}/{p^{a_i-b_i}\mathbb{Z}}  \oplus G
\]
where $b_j=1$ for exactly one index $j$ and $b_i = 0$ for all other indices $i\neq j$. Let $j$ be the index such that $b_j = 1$. Then we can write
\begin{align}
  E\left(\Z/p^n\Z\right) &\cong H \oplus \mathbb{Z}/{p^{a_j}\mathbb{Z}} \\
  E\left(\Z/p^{n-1}\Z\right) &\cong H \oplus \mathbb{Z}/{p^{a_j-1}\mathbb{Z}}
\end{align}
where $H \cong \bigoplus_{i \neq j} \mathbb{Z}/{p^{a_i}\mathbb{Z}} \oplus G$ . Let   $\psi: H \oplus \mathbb{Z}/{p^{a_j}\mathbb{Z}} \rightarrow H \oplus \mathbb{Z}/{p^{a_j-1}\mathbb{Z}}$ be given by
\begin{align}
  &\psi: \left(y,z\right) \mapsto \left(y, z\bmod{p^{a_j-1}}\right)
\end{align}
where $y \in H$, $z \in \Z/p^{a_j}\Z$.
Let $Q\in E\left(\mathbb{Z}/p^n\mathbb{Z}\right)$ be a non-zero point, and let $Q \in H \oplus \mathbb{Z}/{p^{a_j}\mathbb{Z}}$. Then $\psi\left(Q\right)$ is a non-zero point. Assume that there exists a strongly non-zero point $P'\in E\left(\Z/p^{n-1}\Z\right)$ such that $kP'=\sigma_{n,n-1}\left(Q\right)$. Write $Q \cong \left(r,s\right)$ with $r \in H$ and $s \in \mathbb{Z}/{p^{a_j}\mathbb{Z}}$. Similarly write $P' \cong \left(h,g\right)$ with $h \in H$ and $g \in \mathbb{Z}/{p^{a_j-1}\mathbb{Z}}$. By assumption, $kh=r$ and $kg \equiv s \bmod{p^{a_j - 1}}$. Consider the polynomial $f\left(x\right)=kx-s$. Since $k \neq 0 \bmod p$, $f\left(x\right)$ does not have any double roots$\mod p^{a_j-1}$.  Then by Hensel's lemma there exists a number $g' \in \Z/p^{a_j}\Z$ with $g'\equiv g\mod{p^{a_j-1}}$ such that $f\left(g'\right) = 0 \bmod p^{a_j}$. It follows that $kg'-s=0 \bmod p^{a_j}$. Thus $k\left(h,g'\right)=\left(r,s\right)$. Choose $P$ such that $P \cong \left(h,g'\right)$ with $h \in H$ and $g' \in \mathbb{Z}/{p^{a_j-1}\mathbb{Z}}$ yields $kP=Q$.\\
Note that 
\[
  \psi\left(P\right)=\psi\left(h,g'\right) = \left(h,g\right) = P'
\]
Thus by Lemma \ref{lemma:order3} $P$ is a strongly non-zero point.
\end{proof}

\begin{theorem}\label{lemma:order5}
Let $p$ be an odd prime and $Q$ be a non-zero point in $E\left(\Z/p^n\Z\right)$. There exists an integer $k$ and a strongly non-zero point $P\in E\left(\Z/p^n\Z\right)$ such that $kP = Q$ if and only if one of the following holds:

\begin{itemize}
	\item [(a)] $E\left(\Z/p\Z\right)$ is not anomalous.
	\item [(b)] $E\left(\Z/p^n\Z\right) \cong \Z/{p^n\Z}$.
	\item [(c)]  $E\left(\Z/p^n\Z\right) \cong \Z/p\Z \oplus \Z/{p^{n-1}\Z}$ and $Q \cong (Q_1,Q_2)$ with $Q_1 \in \Z/p\Z$, $Q_2 \in \Z/{p^{n-1}\Z}$, where $Q_2$ is not a generator of $\Z/{p^{n-1}\Z}$
\end{itemize}
\end{theorem}
\begin{proof}
The cases when $E$ is not anomalous will be proven by induction on $n$. Note that this is trivially satisfied for $n=1$ because there are no non-zero points and since the order of the curve is coprime to $p$, $k$ is coprime to $p$.
Suppose that the statements holds up to $n-1$. Let $Q$ be a non-zero point in  $E\left(\Z/p^{n}\Z\right)$, and $\sigma_{n,n-1}$ be as defined above. Then $\sigma_{n,n-1}\left(Q\right) \in E\left(\Z/p^{n-1}\Z\right)$ is a non-zero point, so by the inductive hypothesis we have $k,P'$ such that $kP' = \sigma_{n,n-1}\left(Q\right)$ with $k$ coprime to $p$ and $P'$ a strongly non-zero point. Then the claim follows by Lemma \ref{lemma:order7}.
In the case $E\left(\Z/p\Z\right)$  is anomalous at $p$ we have two cases:
\begin{enumerate}[\text{Case} (1):]
\item $E(\Z/{p^n\Z}) \cong \Z/{p^n\Z}$. Consider the natural homomorphism $\sigma_{n,1}: E\left(\Z/p^n\Z\right) \rightarrow E\left(\Z/p\Z\right)$. Since $\sigma_{n,1}$ is surjective, for any generator $P$ of $E\left(\Z/p^n\Z\right)$, $\sigma_{n,1}\left(P\right) \neq \mcO$. Thus there are no non-zero points that are generators of $E\left(\Z/{p^n\Z}\right)$. Therefore there is a strongly non-zero point $P$ which is a generator of $E\left(\Z/{p^n\Z}\right)$. Thus for all points $Q$, there exists a strongly non-zero point $P$ with $kP=Q$ for some integer $k$.
\item $E(\Z/{p^n\Z}) \cong \Z/{p\Z} \oplus \Z/{p^{n-1}\Z}$.  It is well known (CITE SOMETHING) that $E_1/E_{n} \cong \mathbb{Z}/p^{n-1}\mathbb{Z}$ and it follows that for any point $P \cong (P_1,P_2)$ in $E(\Z/{p^n\Z})$ where $P_1 \in \Z/{p\Z}$ and $P_2 \in \Z/{p^{n-1}\Z}$, $P$ is a non-zero point if and only if $P_1$ is the identity. We want to show a non-zero point $Q \cong (0,Q_2)$ can be written as $kP$ for some integer $k$ and some strongly non-zero point $P$ if and only if $Q_2$ is not a generator of $\Z/{p^{n-1}\Z}$.\\
($\Rightarrow$) Assume that $Q \cong (0,Q_2)$ where $Q_2 \in \Z/{p^{n-1}\Z}$ can be written as $kP$ for some integer $k$ and some strongly non-zero point $P \cong (P_1,P_2)$ with $P_1 \in \Z/{p\Z}$ and $P_2 \in \Z/{p^{n-1}\Z}$. Since $P$ is strongly non-zero, $P_1 \neq 0$. However, since $kP_1 = 0$ and $P_1 \in \Z/p\Z$, $\order(P_1) = p$, so $p \mid k$, which implies that $p \mid kP_2 = Q_2$. Therefore $Q_2$ is not a generator of $\Z/{p^{n-1}\Z}$.\\

Conversely, assume $Q \cong (0,Q_2)$ where $Q_2 \in \Z/{p^{n-1}\Z}$ and $Q_2$ is not a generator.  Then $Q_2 = p \cdot r$ for some $r \in \Z/{p^{n-1}\Z}$, and thus 
$Q=kP$ for $k=p$ and $P \cong (1,r)$.
\end{enumerate}
\end{proof}

Note that the last case in Theorem \ref{lemma:order5} only applies for non-zero points that are not generators of the subgroup $\Z/p^{n-1}\Z$ of $E(\Z/p^n\Z)$. The following holds for all non-zero points in $E(\Z/p^n\Z)$.

\begin{lemma}\label{lemma:order9}
Let $E\left(\Z/p^n\Z\right)$ be an elliptic curve and $Q\in E\left(\Z/p^n\Z\right)$ be a non-zero point. Then there exists a strongly non-zero point $P \in E(\Z/p^n\Z)$ such that $\vert Q\vert$ divides $\vert P\vert$.
\end{lemma}
\begin{proof} We will consider the following cases
\begin{enumerate}[\text{Case} (1):]
  \item $Q$ and $E$ satisfy one of the conditions from Theorem \ref{lemma:order5}. In this case $\vert Q\vert \Big | \vert P\vert$ since $kP=Q$.
  \item $E\left(\Z/{p^n\Z}\right) \cong \Z/p\Z \oplus \Z/{p^{n-1}\Z}$ and $Q=(0,Q_2)$ with $Q_2 \in \Z/{p^{n-1}\Z}$, where $Q_2$ is a generator of $\Z/p^{n-1}\Z$. In the case when $n=1$, this case is trivially true because $Q = \mcO \in E(\Z/p^n\Z)$, so $\vert Q\vert = 1$. When $n>1$, note that the order of any point in $E(\Z/p^n\Z)$ divides $p^{n-1}$. There are 
\[
  p \cdot \left(p^{n-1}-p^{n-2}\right)=p^{n-1}\left(p-1\right)
\]
elements with order exactly $p^{n-1}$. Thus there are more than $p^{n-1}$ elements with order $p^{n-1}$. Since there are only $p^{n-1}$ non-zero points, there is a strongly non-zero point $P$ with $\vert P \vert=p^{n-1}$ and thus for any point $Q$, $\vert Q\vert$ divides $\vert P \vert$.
\end{enumerate}
\end{proof}

\begin{lemma}\label{lemma:order10} 
 If $Q$ is a non-zero point in $E\left(\Z/n\Z\right)$, then there exists a strongly non-zero point $P$ such that $\vert Q\vert$ divides $\vert P\vert$.
\end{lemma}
\begin{proof}
 Let $Q\in E\left(\Z/n\Z\right)$ be a non-zero point. Recall that
\[
	E\left(\mathbb{Z}/{n\mathbb{Z}}\right) \cong \bigoplus_{p|n} E\left(\mathbb{Z}/{p^{\ord_p\left(n\right)}\mathbb{Z}}\right)   
\]
Thus each point $T \in E(\Z/n\Z)$ can be written as $T \cong \left(T_{p_1}, T_{p_2},...T_{p_r}\right)$ where $T_{p_i}$ denotes the point corresponding to $T$ in the subgroup $E\left(\Z/p_i^{\ord_{p_i}(n)}\Z\right)$, and $p_1, p_2, \ldots, p_r$ are the distinct prime divisors of $n$. Due to the direct sum, we have that $\vert Q \vert = \lcm \{ \vert Q_{p_i}\vert: 1 \le i \le r \}$.
For each $1 \le i \le r$, let
\[
P_{p_i} = \begin{cases}
Q_{p_i} \text{ if } Q_{p_i} \text{ is a strongly non-zero point }\\
T_{p_i} \text{ where }T_{p_i} \text{ is a strongly non-zero point and } \order\left(Q_{p_i}\right)|\order\left(T_{p_i}\right) 
\end{cases}
\] 
By Corollary \ref{lemma:order2}, $P$ is a strongly non-zero point. Note that $\vert Q \vert \Big |\vert P  \vert$ since $\vert {Q_{p_i}} \vert \Big | \vert P_{p_i}\vert$ for all $1 \le i \le r$ by construction.
\end{proof}
\begin{corollary}\label{lemma:order12}
Let $E/\Q$ be an elliptic curve. A composite number $N$ is an elliptic G-Carmichael number for the curve $E$ if and only if $N$ is a elliptic G-pseudoprime for all strongly non-zero points $P \in E(\Z/N\Z)$.  Similarly, $N$ is an elliptic S-Carmichael number for the curve $E(\Z/N\Z)$ if and only if $N$ is a S-pseudoprime for all strongly non-zero points $P \in E(\Z/N\Z)$.
\end{corollary}
\begin{proof}
We will prove the statement for elliptic G-Carmichael numbers. The proof for elliptic S-Carmichael numbers is similar.  Suppose $N$ is an elliptic G-Carmichael number for a curve $E$ i.e. $N$ is an elliptic G-pseudoprime for all strongly non-zero points $P \in E(\Z/N\Z)$. \\
Conversely,  assume that $N$ is an elliptic G-pseudoprime for all strongly non-zero points $P \in E(\Z/N\Z)$. Then for all strongly non-zero points $P \in E(\Z/N\Z)$, the order $\vert P\vert \Big |  N+1$. By Lemma \ref{lemma:order10}, for any non-zero point $Q$, there exists a strongly non-zero point $P'$ such that 
\[
\vert Q\vert \Big | \vert P'\vert \Big | N+1 
\]
Thus $N$ is an elliptic G-pseudoprime for all points $P \in E(\Z/N\Z)$ i.e  $N$ is an elliptic G-Carmichael number.
\end{proof}

\begin{corollary}\label{lemma:order13}
Let $E/\Q$ be an elliptic curve, $N$ be a composite integer, and $t$ be any integer. Then $\epsilon_{N,p}\left(E\right) \mid t$ if and only if for all strongly non-zero points $P \in E\left(\mathbb{Z}/N\mathbb{Z}\right)$, $\vert P\vert \Big | t$.
\end{corollary}
\begin{proof}
From \cite{REU 2017}, we have $\epsilon_{N,p}\left(E\right) \Big | t$ if and only if for all points $P \in E\left(\mathbb{Z}/{N\mathbb{Z}}\right)$, $tP=\mcO$. By Corollary \ref{lemma:order12},

this is true if and only if for all strongly non-zero points $P \in E\left(\mathbb{Z}/{N\mathbb{Z}}\right)$, $tP=\mcO$.
\end{proof}

\begin{theorem}\label{lemma:order14}
  Let $E/\mathbb{Q}$ be an elliptic curve. There is no composite number $N$ such that $N$ is a strong eliiptic G-pseudoprime for all strongly non-zero points $P \in E(\mathbb{Z}/N\mathbb{Z})$.
\end{theorem}

\begin{proof}
 We prove the claim of the theorem by considering several cases.
\begin{enumerate}[\text{Case} (1):]
  \item $N$ contains a square, i.e. $p^e || N$ for a prime $p$ and an integer $e>1$. 
 We know that $\vert E\left(\Z/p^e\Z\right)\vert = p^{e-1}\vert E(\Z/p\Z)\vert$.  Note that $E(\Z/p^e\Z)$ contains a point of order $p$.  By Lemma \ref{lemma:order9},  there exists a strongly non-zero point $P \in E(\Z/p^e\Z)$ such that $\vert X \vert \mid \vert P \vert$, so $p \Big | \vert P \vert$. In particular, $\vert P \vert \nmid N + 1$, so $\left(N+1\right)P \neq \mcO$. Therefore $N$ is not a strong elliptic G-pseudoprime for the point $P \in E(\Z/p^e\Z)$.
In the following two cases let $p$ be a prime with $p \mid N$ and $\vert E\left(\Z/p\Z\right)\vert = p+1$. Note that such a prime $p$ must exist from our definition of a G-pseudoprime. Also assume $N$ is squarefree.
  
\item There exists a prime $q \neq p$ such that $\vert E\left(\Z/q\Z\right)\vert$ is not a power of $2$.  By the first Sylow theorem, there exists a point $Q\in E\left(\Z/q\Z\right)$ of odd order and there exists a point $P\in E\left(\Z/p\Z\right)$ of even order. Note that the points $P$ and $Q$ are both strongly non-zero points. Write $E\left(\Z/N\Z\right) = E\left(\Z/p\Z\right) \oplus E\left(\Z/q\Z\right) \oplus H$ for some group $H$. Take the point $X = \left(P,Q,h\right)$ for any strongly non-zero element $h \in H$.  If $\left(N+1\right)X = \mcO$, then $n$ is not a strong G-pseudoprime at $X$. Otherwise, (letting $N + 1 = 2^s t$, where $t$ is odd) we must have $tQ = \mcO$ in $E\left(\Z/q\Z\right)$ since $Q$ has odd order. But $tP \neq \mcO \in E(\Z/p\Z)$ because $P$ has even order, so $tX \neq \mcO \in E(\Z/N\Z)$. Because $tQ = \mcO \bmod{q}$, $tX$ is not strongly nonzero and thus $2^rtX$ cannot have the form $\left(x: 0: 1\right)$ for some $x \in \Z/N\Z$. Therefore $N$ is not a strong G-pseudoprime at $X$.

\item For all primes $q \mid N$ with $q \neq p$, $|E\left(\Z/q\Z\right)|$ is a power of $2$. If $|E\left(\Z/p\Z\right)|$ is not a power of two, by the first Sylow theorem there exists a point $P$ of odd order $>1$ in $|E\left(\Z/p\Z\right)|$. Then we can construct a point $X$ as in Case $2$ with $P$ and a point $Q$ of even order from $|E\left(\Z/q\Z\right)|$ for some $q \mid N$. For the rest of the section assume $|E\left(\Z/p\Z\right)|$ is a power of $2$.
Since by Definition \ref{def:strong G-psp}, all prime factors of $N$ must be $\geq 5$, we have that $|E\left(\Z/p\Z\right)| = p + 1 \geq 8$.
Recall that the structure of an elliptic curve over a finite field is the product of two cyclic groups.  Therefore one of the cyclic groups must contain at least $\sqrt{8}>2$ elements and divide a power of $2$. It follows that we can find of point $P$ of order $4$ in $E\left(\Z/p\Z\right)$.

Let $q \mid N$ be a prime, $q \neq p$. Since
$\vert E\left(\Z/q\Z\right)\vert$ is a power of $2$ there exists a point $Q\in E\left(\Z/q\Z\right)$ of order $2$. Write $E\left(\Z/N\Z\right) = E\left(\Z/p\Z\right) \oplus E\left(\Z/q\Z\right) \oplus H$ for some group $H$. Take the strongly nonzero point $X = \left(P,Q,h\right)$ for any strongly nonzero $h \in H$. Then $2tX \neq \mcO$ since $4 \mid \vert P \vert$, but $2tX$ is not strongly nonzero since $2Q = \mcO$. Thus $2^rtX$  cannot have the form $\left(x: 0: 1\right)$ for some $x \in \Z/N\Z$. Therefore $N$ is not a strong G-pseudoprime at $X$.
  \end{enumerate}
\end{proof}

\begin{theorem}\label{lemma:order15}
Let $E/\mathbb{Q}$ be an elliptic curve. Then an odd composite number $N$ is a strong S-pseudoprime for all strongly non-zero points $P \in E\left(\Z/{N\Z}\right)$ if and only if $E$ has good reduction at $p$ for every prime $p \mid N$ and either 
\begin{enumerate}[(i)]
  \item $\epsilon_{N,p}\left(E\right) \mid t$ for all primes $p \mid N$ or
  \item $\epsilon_{N,p}\left(E\right) \mid 2t$ and $E\left(\Z/p\Z\right) \cong \Z/2\Z \oplus \Z/2\Z$ or $\Z/2\Z$ for all primes $p \mid N$.
\end{enumerate}
\end{theorem}
\begin{proof}

\cite{REU 2017} show condition (i) is equivalent to $N$ being a strong S-Carmichael number  and therefore for all strongly non-zero points $P \in E\left(\Z/{N\Z}\right)$, $N$ is a strong S-pseudoprime.
We will prove $N$ is not a strong S-Carmichael number for a curve $E$ and $N$ is a strong S-pseudoprime at all strongly nonzero points $P \in E$ if and only if condition (ii) is met. For notational purposes let $P \cong\left(P_{p_1},P_{p_2}...\right)$ represent the decomposition of the point $P \in E\left(\Z/N\Z\right)$ into $\bigoplus_{p \mid N} E\left(\Z/p_i^{\nu_{p_i}\left(N\right)}\Z\right)$ with $P_{p_i} \in E\left(\Z/p_i^{\nu_{p_i}\left(N\right)}\Z\right)$.
\\
 Let $\epsilon_{N,p}\left(E\right) \mid 2t$ and $E\left(\Z/p\Z\right) \cong \Z/2\Z \oplus \Z/2\Z$ or $\Z/2\Z$ for all primes $p \mid N$. Notice $N$ cannot be a strong S-Carmichael number since there exists a point in $E(\Z/p^{\nu_p(N)}\Z)$ of order divisible by $2$ and therefore $\epsilon_{N,p}\left(E\right) \nmid t$. For any strongly non-zero point $P \cong\left(P_{p_1},P_{p_2} \ldots \right)$, consider $P_{p_i} \in E\left(\Z/p_i^{\ord_{p_i}\left(N\right)}\Z\right)$. Let $\sigma:E\left(\Z/p_i^{\ord_{p_i}\left(N\right)}\Z\right) \rightarrow E\left(\Z/p_i\Z\right)$ be the natural homomorphism. Since $P_{p_i}$ is strongly nonzero, $\sigma\left(P_{p_i}\right)$ is strongly nonzero and thus $\vert \sigma\left(P_{p_i}\right)\vert=2$. Thus $\sigma\left(tP_{p_i}\right) = t\sigma\left(P_{p_i}\right) = \sigma\left(P_{p_i}\right) \neq \mcO$ and so $tP_{p_i}$ is a strongly non-zero point for all $p_i \mid N$. Thus $tP$ is a strongly non-zero point in $E\left(\Z/N\Z\right)$. Since $\epsilon_{N,p}\left(E\right) \mid 2t$ for all $p \mid N$, we have $2tP = \mcO$. Therefore $tP$ must be a strongly non-zero point of order $2$. Therefore $N$ is a strong S-pseudoprime for $P$.
\\\\
Conversely, let $N$ be a strong S-pseudoprime for all strongly nonzero points in $E\left(\Z/{N\Z}\right)$ and $N$ not be a strong S-Carmichael number for $E$. Then there exists a point $Q \in E(\Z/N\Z)$ such that $tQ \neq \mcO$. By Lemma \ref{lemma:order10}, there exists some strongly non-zero point $P$ with $\order\left(Q\right)\mid \order\left(P\right)$, thus $tP \neq \mcO$. Then, by assumption that $N$ is a strong S-pseudoprime for all strongly nonzero points $P \in E\left(\Z/{N\Z}\right)$, $2^{r}tP = \left(x: 0: 1\right)$ for some $0 \leq r < s$ and some $x \in \Z/N\Z$. Fix some strongly nonzero point $P \in E\left(\Z/{N\Z}\right)$.

Assume that $r>0$ and let $T=2^{r}tP$. Note that $T_{p_i} = 2^{r}t{P_{p_i}} \in E(\Z/p_i^{\nu_{p_i}(N)}\Z)$ is a strongly nonzero point of order $2$ in $E(\Z/p_i^{\nu_{p_i}(N)} \Z)$ for all $p_i \mid N$ and $2T = \mcO \in E(\Z/N\Z)$. Construct the strongly nonzero point $Y \cong \left(T_{p_1},P_{p_2}, \ldots, P_{p_i},\ldots\right)$. Notice $2^{r+1}tY = \mcO$, $2^{r}tY \neq \mcO$, and $2^{r}tY$ is a non-zero point since $2T_{p_1}=\mcO$ in $E\left(\Z/p_1^{\ord_{p_1}\left(N\right)}\Z\right)$, thus $N$ is not a strong elliptic S-pseudoprime for the strongly non-zero point $Y$, a contradiction. Therefore $r=0$. 

Assume for the sake of contradiction there exists a strongly non-zero point $X$ such that $tX=\mcO$. Construct the strongly nonzero point $Y' \cong \left(X_{p_1},P_{p_2}, \ldots, P_{p_i},\ldots \right)$. Notice $tY' \neq \mcO$ since $tP_{p_2} \neq \mcO$. Since $2^itX_{p_1} = \mcO$ for all $i \geq 0$, $2^i t Y'$ is not a strongly non-zero point for all $i \ge 0$. Thus $N$ is not strong elliptic S-pseudoprime at the point $Y'$.
Therefore for all strongly non-zero points $P \in E\left(\Z/N\Z\right)$, $tP \neq \mcO$.\\

Thus for every strongly non-zero point $P\cong (P_{p_1},P_{p_2}..)$, $2tP = \mcO$ and so $\epsilon_{N,p} \mid 2t$ . We also conclude that $\ord_2(\vert P_{p_i}\vert) = 1$. Let $\sigma:E\left(\Z/{p_i^{\ord_{p_i}\left(N\right)}\Z}\right) \rightarrow E\left(\Z/p_i\Z\right)$ be the natural homomorphism. Since $ker(\sigma) = \mathbb{Z}/p^{\ord_{p_i}-1}\mathbb{Z}$ and $p_i$ is odd, 
$\ord_2(\vert \sigma(P_{p_i}) \vert) = \ord_2(\vert P_{p_i} \vert) = 1$. Since every nonzero point in $X \in E(\mathbb{Z}/p_i\mathbb{Z}$ can be lifted to a strongly nonzero point $X' \in E(\mathbb{Z}/N\mathbb{Z})$ such that $\sigma(X') = X$, $\ord_2(X) = 1$. 

Thus, there are no non-zero points in $E(\mathbb{Z}/p\mathbb{Z})$ have odd order $>1$ or order $4$ thus every non-zero point has order $2$. Therefore $E\left(\Z/p\Z\right) \cong \Z/{2\Z} \oplus \Z/2\Z$ or $E\left(\Z/p\Z\right) \cong \Z/2\Z$.
\end{proof}

\section{Point-wise Probabilities for Strong Elliptic Pseudoprimes}

It is known that no composite number can be a strong G-pseudoprime for all points on a given elliptic curve (see \cite{Ekstrom}), so we now ask: for how many of the points on a given elliptic curve can a given composite number be a strong G-pseudoprime? Similarly, no composite number can be a strong S-pseudoprime for all points on all elliptic curves, which motivates the following question: given a fixed composite number $N$, what is the probability that $N$ is a strong elliptic S-pseudoprime for a randomly chosen point on a randomly chosen elliptic curve?
\begin{theorem}\label{strong gpseudoprime}
A composite number $N$ is a strong elliptic G-pseudoprime for at most $5/8$ of the points in $E\left(\mathbb{Z}/N\mathbb{Z}\right)$.
\end{theorem}

\begin{notation*}
    Let $E/\Q$ be an elliptic curve, and let $N$ be a positive integer such that $E$ has good reduction at $p$ for every prime $p \mid N$. For a positive integer $x$ and a prime power $p^a$, define
    \[
        J(x,p^a) := \frac{\# \{ P \in E(\Z/p^a\Z): \nu_2(\order(P)) = x \} }{\# E(\Z/p^a\Z)},
    \]
    where $\nu_2$ denotes 2-adic valuation.
    Let $N = p_1^{a_1}\cdots p_k^{a_k}$ be the prime factorization of $N$, so we may write $E(\Z/N\Z) \cong E(\Z/p_1^{a_1}\Z) \oplus \cdots \oplus E(\Z/p_k^{a_k}\Z)$. For each $P \in E(\Z/N\Z)$, we can consider $P$ as isomorphic to $(P_1,\ldots, P_k)$, where $P_i \in E(\Z/p_i^{a_i}\Z)$. Define 
    \begin{align}
    H(x,N) := \frac{\# \{ P \in E(\Z/N\Z): \nu_2(\order(P_i)) = x \text{ for all } 1 \le i \le k\}}{\# E(\Z/N\Z)} = \prod_{i=1}^{k} J(x,p_i^{a_i})
    \label{eq:H(x,N)}
    \end{align}
    where $\order(P_i)$ denotes the order of $P_i$ as an element of $E(\Z/p_i^{a_i}\Z)$.
\end{notation*}

\begin{proposition}\label{distinct}
Let $N=p_1^{\alpha_1}...p_k^{\alpha_k}$ and $M = N/p_k^{\alpha_k}$. Define $G\left(E,N\right) = \sum_{x \ge 0} H(x,N)$. Then $G\left(E,N\right) \leq G\left(E,M\right)$.
\end{proposition}
\begin{proof}
\begin{align*}
G(E,N) = \sum_{x} H(x,N) = \sum_{x} J(x,p_k^{a_k})H(x,M) \leq \sum_{x} H(x,M) = G(E,M).
\end{align*}
\end{proof}
\begin{lemma}\label{powermin}
Let $E/\Q$ be an elliptic curve with good reduction at a prime $p > 2$. Then for any $x \ge 0$ and any $a \in \N$, $J(x,p^a) = J(x,p)$. 
\end{lemma}
\begin{proof}
We prove the claim of the theorem by considering several cases.
\begin{enumerate}[\text{Case} (1):] 
\item $E\left(\mathbb{Z}/p_k^{\alpha_k}\mathbb{Z}\right) \cong \mathbb{Z}/p_k^{\alpha_k-1}\mathbb{Z} \oplus E\left(\mathbb{Z}/p_k\mathbb{Z}\right)$.\\

Take any point $P \cong \left(P_1, P_2\right)$ in $E\left(\mathbb{Z}/p_k^{\alpha_k}\mathbb{Z}\right) \cong \mathbb{Z}/p_k^{\alpha_k-1}\mathbb{Z} \oplus E\left(\mathbb{Z}/p_k\mathbb{Z}\right)$. Note that the $\order\left(P\right) = \lcm\left(\order\left(P_1\right), \order\left(P_2\right)\right)$. $2 \nmid p_k^{\alpha_k-1}$, so $\nu _2\left(\order\left(P\right)\right) = \nu _2\left(\order\left(P_2\right)\right)$.
 
\item $E\left(\mathbb{Z}/p_k^{\alpha_k}\mathbb{Z}\right) \cong \mathbb{Z}/p_k^{\alpha_k}\mathbb{Z}$.\\
Since $p_k$ is anomalous, $E\left(\mathbb{Z}/p_k\mathbb{Z}\right) \cong \mathbb{Z}/p_k\mathbb{Z}$. Since $p_k$ is odd, all of the points in $E\left(\mathbb{Z}/p_k\mathbb{Z}\right)$ have odd order, and all of the points in $E\left(\mathbb{Z}/p_k^{\alpha_k}\mathbb{Z}\right)$ have odd order.
\end{enumerate}
\end{proof}

\begin{corollary}
    Let $N = p_1^{a_1}\cdots p_k^{a_k}$ with $p_i \neq 2$, and define $M = p_1p_2\cdots p_k$. Then for all $x \ge 0$, $H(x,N) = H(x,M)$.
\end{corollary}
The statement follows from Lemma \ref{powermin} and the definition of $H(x,N)$ in \ref{eq:H(x,N)}. As a result, we may now assume without loss of generality that $a_i = 1$ for all $i$.

 \begin{lemma} \label{lemma:pointproportion}
Consider the group $G \cong \mathbb{Z}/2^st\mathbb{Z} \oplus \mathbb{Z}/2^rw\mathbb{Z}$, where $t$ and $w$ are odd and $2^st \mid 2^rw$. The proportion of points $P\in G$ such that $\nu _2\left(\vert P\vert \right) = k$ is 
\begin{enumerate}[(i)]
\item $1/2^{r+s}$ if $k = 0$,
\item $3\left(2^{2k-2}\right)/2^{r+s}$ if $1 \leq k \leq s$,
\item $2^{s + k -1}/2^{r+s}$ if $s+1 \leq k \leq r$, and
\item $0$ for $k > r$.
\end{enumerate}
\end{lemma}
\begin{proof}
First let us deal with the proportion of points $P \in G$ of odd order, i.e. when $\nu_2(\vert P\vert) = 0$. Using the isomorphism, we can consider $P \in G$ as the pair of points $(P_1, P_2)$, where $P_1 \in \Z/2^s t\Z$ and $P_2 \in \Z/2^r w\Z$. We can consider $P_1$ to be an integer modulo $2^st$ in the set $\{0,1,2,\ldots, 2^st -1\}$. 
\\\\
\textit{Claim 1.} The order of $P_1$ is odd if and only if $2^s \mid P_1$.
\\
\textit{Proof of Claim 1.} Suppose that $P_1$ has odd order $d$. Then $dP_1 \equiv 0 \bmod{2^st}$. In particular, $2^{s}t \mid dP_1$. Since $d$ is odd, $2^{s} \mid P_1$. \\
Conversely, suppose $2^s \mid P_1$. Let $d$ be the order of such an element $P_1$. Then $dP_1 \equiv 0 \bmod{2^{s}t}$, i.e. $2^s t \mid dP_1$. If $d$ were even, then we still have that $2^s t \mid (d/2)P_1$, since $2^s \mid P_1$, contradicting the minimality of $d$. Therefore $d$ must be odd, concluding the proof of Claim 1.
\\\\
Therefore to count the points of odd order in $\Z/2^st\Z$, it suffices to count the elements $P_1$ in the set $\{0,1,2,\ldots, 2^st - 1\}$ such that $2^s \mid P$. These are exactly the elements $\{0, 2^s, 2 \cdot 2^s, \ldots, (t-1)\cdot 2^s\}$, of which there are $t$. Because an element of odd order in $G$ must correspond to elements of odd order in both $\Z/2^st \Z$ and $\Z/2^rw\Z$, we can see that there are $tw$ points of odd order, including the identity. Dividing by $|G| = 2^{r+s}tw$ gives us part (i) of the lemma.
\\\\
We can also extend this to counting points $P_1$ of order $d$ such that $\nu_2(d) = k$ for some $1 \le k \le s$:
\\\\
\textit{Claim 2.} Suppose that the order of $P_1 \in \Z/2^st\Z$ is $d$. Then $\nu_2(d) = k$ if and only if $2^{s-k} \mid \mid P_1$.
\\
\textit{Proof of Claim 2.} Suppose that $\nu_2(d) = k$. As in the previous claim, we have $dP_1 \equiv 0 \bmod{2^st}$, so $2^st \mid dP_1$, so $2^s \mid dP_1$. By the minimality of $d$, $2^st \nmid (d/2)P_1$. Therefore $2^{s} \mid \mid dP_1$. Since we know that $2^k \mid \mid d$, we must have that $2^{s - k} \mid \mid P_1$.\\
Suppose instead that $2^{s-k} \mid \mid P_1$. Let the order of $P_1$ be $d$. Then $2^st \mid dP_1$, so $2^k t \mid d\cdot P_1/2^{s-k}$. Since $P_1/2^{s-k}$ is odd, we must have that $2^k \mid d$. Furthermore, we cannot have that $2^{k+1} \mid d$, otherwise $2^st \mid (d/2)P_1$, contradicting the minimality of $d$. Therefore $\nu_2(d) = k$, concluding the proof of Claim 2.
\\\\
From Claim 2, we see that to count the points $P_1 \in \Z/2^st\Z$ with order $d$ such that $\nu_2(d) = k$ for some $1 \le k \le s$, it suffices to count the number of multiples of $2^{s-k}$ in the set $\{0,1,2,\ldots, 2^st - 1 \}$ which are \emph{not} multiples of $2^{s-k+1}$. There are exactly $2^st/2^{s-k} - 2^{s}t/2^{s-k+1}= 2^{k-1}t$ of these. Furthermore, there are exactly $2^{k}$ points $P_1$ such that $\nu_2(\order(P_1)) \le k$.

Let $k$ be an integer such that $0 \le k \le s$. Note that
\begin{align*} 
	&\#\{P \in G  \hspace{0.1cm}|\hspace{0.1cm} \nu_2(\order(P)) = k \} \\
	&\hspace{0.5cm}=\hspace{0.1cm}\#\{P_1 \in \Z/2^st\Z \hspace{0.1cm}|\hspace{0.1cm} \nu_2\vert P_1\vert = k\} \times \#\{P_2 \in \Z/2^rw\Z \hspace{0.1cm}|\hspace{0.1cm} \nu_2\vert P_2\vert \le k\} \\
    &\hspace{0.5cm}+\hspace{0.1cm} \#\{P_1 \in \Z/2^st\Z \hspace{0.1cm}|\hspace{0.1cm} \nu_2\vert P_1\vert  \le k\} \times \#\{P_2 \in \Z/2^rw\Z \hspace{0.1cm}|\hspace{0.1cm} \nu_2\vert P_2\vert  = k\} \\
    &\hspace{0.5cm}- \hspace{0.1cm}\#\{P_1 \in \Z/2^st\Z \hspace{0.1cm}|\hspace{0.1cm} \nu_2\vert P_1\vert = k\} \times \#\{P_2 \in \Z/2^rw\Z \hspace{0.1cm}|\hspace{0.1cm} \nu_2\vert P_2\vert = k\}
\end{align*}
This simplifies to the following 
\begin{align*}
	\#\{P \in G  \hspace{0.1cm}|\hspace{0.1cm} \nu_2(\order(P)) = k \} &= 2^{k-1}t\cdot 2^{k}w + 2^{k}t \cdot 2^{k-1} w - 2^{k-1} t \cdot 2^{k-1} w\\
    &= 3\cdot 2^{2k-2}tw.
\end{align*}
Dividing this final expression by the order of the group $|G| = 2^{r+s}tw$ gives us part (ii) of the lemma.
\\\\
We are interested in counting the number of points $P$ such that 
\[
\#\{P \in G  \hspace{0.1cm}|\hspace{0.1cm} \nu_2(\order(P)) = k \}
\]
when $s < k \le r$. Because there are no points $P_1 \in \Z/2^{s}t\Z$ such that $\nu_2(\order(P_1)) = k$, we have
\begin{align*}
	&\#\{P \in G  \hspace{0.1cm}|\hspace{0.1cm} \nu_2\vert P\vert = k \} \\
    &\hspace{0.5cm}=\hspace{0.1cm}\#\{P_1 \in \Z/2^st\Z \hspace{0.1cm}|\hspace{0.1cm} \nu_2\vert P_1\vert  \le k\} \times \#\{P_2 \in \Z/2^rw\Z \hspace{0.1cm}|\hspace{0.1cm} \nu_2\vert P_2\vert = k\} \\
    &\hspace{0.5cm}=\hspace{0.1cm}\#(\Z/2^st\Z) \times 2^{k-1}w \\
    &\hspace{0.5cm}=\hspace{0.1cm}2^{s + k - 1}tw.
\end{align*}
Dividing the above expression by $|G| = 2^{r+s}tw$ gives part (iii) of the lemma. Observe that the order of any element $P = (P_1, P_2) \in G$ has an order which is the LCM of the orders of $P_1 \in \Z/2^st \Z$ and $P_2 \in \Z/2^rw \Z$. Furthermore, since $d_1 := \order(P_1) \mid 2^st$ and $d_2 := \order(P_2) \mid 2^rw$, combined with the fact that $2^st \mid 2^rw \Rightarrow s \le r$, we have that $\lcm(d_1, d_2) \mid 2^r$, so $\nu_2(\order(P)) \le r$, which gives us part (iv) of the lemma.
\end{proof}

\begin{definition}
Consider the group $G \cong \mathbb{Z}/2^s t\mathbb{Z} \oplus \mathbb{Z}/2^r w\mathbb{Z}$, where $t$ and $w$ are odd and $2^st \mid 2^rw$. Define the vector $h\left(s,r\right)$ in $\mathbb{R}^{\infty}$ such that the $i^{th}$ coordinate of $h(s,r)$ is the proportion of points $P \in G$ such that $\nu _2\left(P\right) = i$.
\end{definition}
The proof of the following theorem involves a large amount of casework and computation and it is included in the Appendix section.
\begin{theorem} \label{thm:max h}
Suppose $s_1\le r_1$ and $s_2\le r_2$ and $r_1 \geq 1$. Then $\vec{h}(s_1, r_1)\cdot \vec{h}(s_2, r_2)$ is maximized when $r_1=r_2=s_1=s_2=1$.
\end{theorem}

\begin{lemma}\label{incompatpoints}
Let $p$ be a prime, $\alpha \geq 2$, and $E(\mathbb{Z}/p^{\alpha}\mathbb{Z})$ be an elliptic curve. The proportion of points in $E(\mathbb{Z}/p^{\alpha}\mathbb{Z})$ with order divisible by $p$ is at least $\frac{p^{\alpha-1}-1}{p^{\alpha-1}}$.
\end{lemma}
\begin{proof}
 We prove the claim of the theorem by considering several cases.
\begin{enumerate}[\text{Case} (1):]
\item p is anomalous for E and $E(\mathbb{Z}/p^{\alpha}\mathbb{Z}) \cong \mathbb{Z}/p^{\alpha}\mathbb{Z}$\\ Every element has order $p^k$ for some $0 \leq k \leq \alpha$. The only element with order 1 is the identity, so the proportion of points in E with order divisible by p is $\frac{p^{\alpha}-1}{p^{\alpha}}$.\\
\item : $E(\mathbb{Z}/p^{\alpha}\mathbb{Z}) \cong \mathbb{Z}/p^{\alpha-1}\mathbb{Z} \oplus E(\mathbb{Z}/p\mathbb{Z})$\\
If $P \cong (P_1, P_2)$ is in $\mathbb{Z}/p^{\alpha-1}\mathbb{Z} \oplus E(\mathbb{Z}/p\mathbb{Z})$, and if the order of $P_1$ is divisible by p, then the order of P is divisible by P. The proportion of points $P = (P_1, P_2)$ in $\mathbb{Z}/p^{\alpha-1}\mathbb{Z} \oplus E(\mathbb{Z}/p\mathbb{Z})$ with p dividing the order of $P_1$ is at least $\frac{p^{\alpha-1}-1}{p^{\alpha-1}}$.
\end{enumerate}
\end{proof}
Finally, we give the proof of Theorem \ref{strong gpseudoprime}.
\begin{proof}{(of Theorem \ref{strong gpseudoprime})}
 We prove the claim of the theorem by considering several cases.
\begin{enumerate}[\text{Case} (1):]
\item Suppose $N$ is not squarefree.
By Lemma \ref{distinct}, the maximum proportion of points that N can be a strong G-pseudoprime for will occur when $N = p^\alpha$. Suppose $N = p^{\alpha}$ and $E\left(\mathbb{Z}/N\mathbb{Z}\right) \cong E\left(\mathbb{Z}/p^{\alpha}\mathbb{Z}\right)$ with $\alpha > 1$. By Lemma \ref{incompatpoints}, the proportion of points in $E\left(\mathbb{Z}/p^{\alpha}\mathbb{Z}\right)$ with order divisible by p is at least $\frac{p^{\alpha-1}-1}{p^{\alpha-1}}$. Since $p \nmid N+1$, if P has order divisible by p, then $\left(N+1\right)P \not \equiv \mathcal{O}$ and N is not a strong G-pseudoprime for (E,P). So the proportion of points in $E\left(\mathbb{Z}/N\mathbb{Z}\right)$ for which $N$ is a strong elliptic G-pseudoprime is at most $\frac{1}{p} < \frac{5}{8}$.

\item Suppose $N$ is squarefree. Since $\left(\frac{-d}{N}\right) = -1$, there exists some prime $p$ such that $\left(\frac{-d}{p}\right) = -1$ thus $|E\left(\mathbb{Z}/p\mathbb{Z}\right)| = p + 1$. Since $N$ is composite and squarefree, there exists a prime $q \mid N$, $q \neq p$.
Suppose $E\left(\mathbb{Z}/p\mathbb{Z}\right) \cong \mathbb{Z}/2^{s_1}t_1\mathbb{Z} \oplus \mathbb{Z}/2^{r_1}w_1\mathbb{Z}$ where $t_1$ and $w_1$ are odd and $2^{s_1}t_1 \mid 2^{r_1}w_1$ and $E\left(\mathbb{Z}/q\mathbb{Z}\right) \cong \mathbb{Z}/2^{s_2}t_2\mathbb{Z} \oplus \mathbb{Z}/2^{r_2}w_2\mathbb{Z}$ where $t_2$ and $w_2$ are odd and $2^{s_2}t_2 \mid 2^{r_2}w_2$. 
N is a strong G-pseudoprime at a point $P = \left(P_1, P_2\right)$, where $P_1 \in \Emod{p}$ and $P_2 \in \Emod{q}$, only if $\nu _2\left(P_1\right) = \nu _2\left(P_2\right)$. The percentage of points that satisfy this is $h\left(s_1,r_1\right) \cdot h\left(s_2, r_2\right)$, with $r_1 \geq 1$ since $|E\left(\mathbb{Z}/p\mathbb{Z}\right)|$ is even. By Theorem \ref{thm:max h}, this percentage is at most $h\left(1,1\right) \cdot h\left(1, 1\right) = 5/8$ of the points in $E\left(\mathbb{Z}/N\mathbb{Z}\right)$.
\end{enumerate}
\end{proof}
\begin{lemma}\label{lemma:Eoverprimechar}
Let $p$ be an odd prime, and let $E/\Q: y^2= x^3+Ax+B$ be an elliptic curve that has good reduction at $p$. Write $\Emod{p} \cong \Zmod{2^st} \oplus \Zmod{2^rw}$, where $t, w>0$ are odd integers and $2^st \mid 2^rw$. Then
\begin{itemize}
\item $s=r=0$ if and only if $x^3+Ax+B$ is irreducible $\mod p$
\item $s = 0$ and $r \geq 1$ if and only if $x^3+Ax+B$ has one root $\mod p$
\item $r \geq s \geq 1$ if and only if $x^3+Ax+B$ has three roots $\mod p$
\end{itemize}
\end{lemma}
\begin{proof}
The points of order 2 on $\Emod{p}$ are exactly the roots of $x^3+Ax+B$ mod $p$. Note this only works for $p$ prime - this statement fails spectacularly for composite numbers. If there are no roots, then the P-Sylow theorems implies that $|\Emod{p}|$ is odd since there are no points of order 2. Hence $a_p$ is odd. If there is one root, there is one point of order 2, so exactly one of $r$ or $s$ must be nonzero. But by assumption $r\ge s$, so $s=0$ and $r\ge 1$. If there are three roots, there are three points of order 2, so both $r$ and $s$ must be at least 1. The converse of each statement in the theorem also holds since every point of order 2 is a root.
\end{proof}

\begin{lemma}\label{primestruct}
Let $E/\Q$ be an elliptic curve with good reduction at an odd prime $p$. Write  $\Emod{p} \cong \Zmod{2^st} \oplus \Zmod{2^rw}$, where $t, w>0$ are odd integers and $2^st \mid 2^rw$. Then
\begin{itemize}
\item $s = r = 0$ with probability $\frac{p+1}{3p}$
\item $s = 0$ and $r \geq 1$ with probability $\frac{1}{2}$
\item $r \geq s \geq 1$ with probability $\frac{p-2}{6p}$.
\end{itemize}
\end{lemma}

\begin{proof}
There are $\binom{p}{3} = \frac{p\left(p-1\right)\left(p-2\right)}{6}$ such curves with 3 roots.  There are $p^2$ quadratic polynomials in $\mathbb{F}_p[x]$, and $\binom{p}{2} + p$ quadratic polynomials with roots in $\mathbb{F}_p$. So there are $p\left(p^2 - \binom{p}{2} - p \right)$ such curves with 1 root. There are $p^3 - p^2 - \binom{p}{3} - p\left(p^2 - \binom{p}{2} - p \right)$ such curves with no roots, and $s=r=0$. There are $p^3$ cubic polynomials, and $p^2$ cubic polynomials with repeated roots. So there are $p^3 - p^2$ possible curves $\Emod{p}$ with good reduction at p. The lemma follows from lemma \ref{lemma:Eoverprimechar} and the counting in the paragraph.
\end{proof}

\begin{theorem} \label{thm:propoverallcurves}
Let $N=p_1^{\alpha_1}...p_k^{\alpha_k}$ with $2 \nmid N$ and $p,q \mid N$. The probability a random point $P \cong \left(P_1, ... , P_k\right)$ has $\nu _2\left(\order\left(P_i\right)\right)$ all equal for a random curve $E\left(\mathbb{Z}/N\mathbb{Z}\right) \cong E\left(\mathbb{Z}/p_1^{\alpha_1}\mathbb{Z}\right) \oplus ... \oplus E\left(\mathbb{Z}/p_k^{\alpha_k}\mathbb{Z}\right)$ is at most $\frac{17pq + 2p + 2q +4}{32pq}$.
\end{theorem}
\begin{proof}
Let $G\left(E,N\right)$ be the proportion of points $P \cong \left(P_1, ... , P_k\right)$ in $E\left(\mathbb{Z}/N\mathbb{Z}\right) \cong E\left(\mathbb{Z}/p_1^{\alpha_1}\mathbb{Z}\right) \oplus ... \oplus E\left(\mathbb{Z}/p_k^{\alpha_k}\mathbb{Z}\right)$ such that $\nu _2\vert P_i\vert$ are all equal. Let $E\left(\mathbb{Z}/p\mathbb{Z}\right) \cong \mathbb{Z}/2^{s_i}t_i\mathbb{Z} \oplus \mathbb{Z}/2^{r_i}w_i\mathbb{Z}$ and $E\left(\mathbb{Z}/q\mathbb{Z}\right) \cong \mathbb{Z}/2^{s_j}t_j\mathbb{Z} \oplus \mathbb{Z}/2^{r_j}w_j\mathbb{Z}$. Let $|E|$ be the number of elliptic curves $E\left(\mathbb{Z}/N\mathbb{Z}\right)$ with good reduction. Let $|E'|$ be the number of elliptic curves $E\left(\mathbb{Z}/p\mathbb{Z}\right) \oplus E\left(\mathbb{Z}/q\mathbb{Z}\right)$ with good reduction at p and q. By Lemma \ref{primestruct} we have 
\begin{itemize}
\item $r_1 = r_2 = s_1 = s_2 = 0$ with probability $\frac{\left(p+1\right)\left(q+1\right)}{9pq}$
\item $r_1 = s_1 = s_2 = 0$ and $r_2 \geq 1$ with probability $\frac{\left(p+1\right)}{6p}$
\item $r_2 = s_2 = s_1 = 0$ and $r_1 \geq 1$ with probability $\frac{\left(q+1\right)}{6q}$
\item $s_1 = r_1 = 0$ and $r_2 \geq s_2 \geq 1$ with probability $\frac{\left(p+1\right)\left(q-2\right)}{18pq}$
\item $s_2 = r_2 = 0$ and $r_1 \geq s_1 \geq 1$ with probability $\frac{\left(q+1\right)\left(p-2\right)}{18pq}$
\item $s_1 = s_2 = 0$, $r_1 \geq 1$, and $r_2 \geq 1$ with probability $\frac{1}{4}$
\item $s_1 = 0$, $r_1 \geq 1$, and $r_2 \geq s_2 \geq 1$ with probability $\frac{\left(p-2\right)}{12p}$
\item $s_2 = 0$, $r_2 \geq 1$, and $r_1 \geq s_1 \geq 1$ with probability $\frac{\left(q-2\right)}{12q}$
\item $r_1 \geq s_1 \geq 1$ and $r_2 \geq s_2 \geq 1$ with probability $\frac{\left(p-2\right)\left(q-2\right)}{36pq}$
\end{itemize}
By Lemma \ref{distinct} and Lemma \ref{powermin},  we have that $G\left(E, N\right) \leq h\left(s_1, r_1\right)\cdot h\left(s_2, r_2\right)$. Then
\begin{align*}
&Pr[\textnormal{A random point $P \cong \left(P_1, ... , P_k\right)$ has $\nu _2\vert P_i\vert$ all equal}\\
&\textnormal{for a random curve $E\left(\mathbb{Z}/N\mathbb{Z}\right) \cong E\left(\mathbb{Z}/p_1^{\alpha_1}\mathbb{Z}\right) \oplus ... \oplus E\left(\mathbb{Z}/p_k^{\alpha_k}\mathbb{Z}\right)$}] =\\
&\frac{1}{|E|}\sum_{E\left(\mathbb{Z}/N\mathbb{Z}\right)} G\left(E, N\right) \leq \\
&\frac{1}{|E'|}\sum_{E\left(\mathbb{Z}/p\mathbb{Z}\right) \oplus E\left(\mathbb{Z}/q\mathbb{Z}\right)} h\left(s_i, r_i\right)\cdot h\left(s_j, r_j\right) \leq \\
&\frac{\left(p+1\right)\left(q+1\right)}{9pq} h\left(0,0\right) \cdot h\left(0,0\right) + \left(\frac{\left(p+1\right)}{6p} + \frac{\left(q+1\right)}{6q}\right) h\left(0,0\right) \cdot h\left(0,1\right)+ \\
&\big(\frac{\left(p+1\right)\left(q-2\right)}{18pq} + \frac{\left(q+1\right)\left(p-2\right)}{18pq}\big) h\left(0,0\right) \cdot h\left(1,1\right) + \frac{1}{4} h\left(0,1\right) \cdot h\left(0,1\right) + \\
&\left(\frac{\left(p-2\right)}{12p} + \frac{\left(q-2\right)}{12q}\right)h\left(0,1\right) \cdot h\left(1,1\right) + \frac{\left(p-2\right)\left(q-2\right)}{36pq}h\left(1,1\right) \cdot h\left(1, 1\right)\\
&= \frac{17pq + 2p + 2q +4}{32pq}\vspace{-0.2in}
\end{align*}
\end{proof}

\begin{corollary}
Let $N=p_1^{\alpha_1}...p_k^{\alpha_k}$ with $2 \nmid N$ and p and q the largest primes dividing N. The probability N is a strong S-pseudoprime for a random point $P \cong \left(P_1, ... , P_k\right)$ on a random curve $E\left(\mathbb{Z}/N\mathbb{Z}\right) \cong E\left(\mathbb{Z}/p_1^{\alpha_1}\mathbb{Z}\right) \oplus ... \oplus E\left(\mathbb{Z}/p_k^{\alpha_k}\mathbb{Z}\right)$ is at most $\frac{17pq + 2p + 2q +4}{32pq}$.
\end{corollary}
\subsection{Strongly Non-zero Point-wise Probabilities}
In this section we prove similar probabilistic results  for strong elliptic G-pseudoprime (strong elliptic S-pseudoprime) and strongly non-zero points on a given elliptic curve. As in Section 3, we will ignore the case where there are no strongly non-zero points. The proof of following theorem will be given at the end of this section. First we will show several results needed to prove the theorem.
\begin{theorem}\label{strong primitive gpseudoprime}
Let $\ENZ$ be an elliptic curve. A composite number $N$ is a strong elliptic G-pseudoprime for at most 9/11 of the strongly non-zero points in $\ENZ$.
\end{theorem}

The following lemma is a direct consequence of Lemma \ref{lemma:pointproportion}. 
\begin{lemma}
Suppose $G \cong \mathbb{Z}/2^st\mathbb{Z} \oplus \mathbb{Z}/2^rw\mathbb{Z}$ where $t$ and $w$ are odd and $2^st \mid 2^rw$. The percentage of strongly non-zero points $P\in G$ such that $\nu _2(order(P)) = x$ is $\frac{tw - 1}{2^{r+s}tw-1}$ for $x = 0$, $\frac{3(2^{2x-2})tw}{2^{r+s}tw-1}$ for $1 \leq x \leq s$, $\frac{2^{s + x-1}tw}{2^{r+s}tw-1}$ for $s+1 \leq x \leq r$, and 0 for $x > r$.
\end{lemma}

\begin{definition}
Define the vector $h'(s,r, t, w)$ to be the vector whose $i^{th}$ coordinate is the percentage of strongly non-zero points $P \in G \cong \mathbb{Z}/2^st\mathbb{Z} \oplus \mathbb{Z}/2^rw\mathbb{Z}$ where $t$ and $w$ are odd and $2^st \mid 2^rw$, such that $\nu _2(P) = i$.
\end{definition}
The proof of the following theorem a large amount of casework and is placed in the Appendix section.
\begin{theorem}\label{strong computation}
 Suppose $s_1 \leq r_1$ and $s_2 \leq r_2$ and $r_1 \geq 1$ and $2^s_2\cdot 2^r_2 \cdot t_2 \cdot w_2 > 1$. Then $h'(s_1,r_1,t_1,w_1) \cdot h'(s_2,r_2,t_2,w_2) \leq h'(1,1,t_1,w_1)h'(1,1,t_2,w_2)$
\end{theorem}

Finally, we give the proof of {{\bf {Theorem }\ref{strong primitive gpseudoprime}} stated at the beginning of this section. 
\begin{proof}
We prove the claim of the theorem by considering several cases.
\begin{enumerate}[\text{Case} (1):] 
\item Suppose $N$ is not squarefree. By Lemma \ref{distinct}, the maximum proportion of strongly non-zero points that N can be a strong G-pseudoprime for will occur when $N = p^\alpha$. Suppose $N = p^{\alpha}$ and $E(\mathbb{Z}/N\mathbb{Z}) \cong E(\mathbb{Z}/p^{\alpha}\mathbb{Z})$ with $\alpha > 1$. 
From the structure theorem for abelian groups, at least $\frac{p^\alpha-1}{p^\alpha}$ of the strongly non-zero points in $E(\mathbb{Z}/p^{\alpha}\mathbb{Z})$ have order p. Since $p \nmid N+1$, if P has order p, then $(N+1)P \not \equiv \mathcal{O}$ and N is not a strong G-pseudoprime for (E,P). So N can be a strong G-pseudoprime for at most $\frac{1}{p^{\alpha-1}} < 
\frac{9}{11}$ of the strongly non-zero points in $E(\mathbb{Z}/N\mathbb{Z})$.
\item Suppose $N$ is squarefree. Since $\left(\frac{-d}{N}\right) = -1$, there exists some prime $p$ dividing $N$ such that $\left(\frac{-d}{p}\right) = -1$ thus $|E(\mathbb{Z}/p\mathbb{Z})| = p+1$. Since $N$ is composite and squarefree there exists a prime $q \mid N$, $q \neq p$. 

 Suppose $E(\mathbb{Z}/p\mathbb{Z}) \cong \mathbb{Z}/2^{s_1}t_1\mathbb{Z} \oplus \mathbb{Z}/2^{r_1}w_1\mathbb{Z}$ where $t_1$ and $w_1$ are odd and $2^{s_1}t_1 \mid 2^{r_1}w_1$ and $E(\mathbb{Z}/q\mathbb{Z}) \cong \mathbb{Z}/2^{s_2}t_2\mathbb{Z} \oplus \mathbb{Z}/2^{r_2}w_2\mathbb{Z}$ where $t_2$ and $w_2$ are odd and $2^{s_2}t_2 \mid 2^{r_2}w_2$. N is a strong G-pseudoprime at a point $P = (P_1, P_2)$, only if $\nu _2(P_1) = \nu _2(P_2)$. The percentage of strongly non-zero points that satisfy this is $h'(s_1,r_1, t_1, w_1) \cdot h'(s_2, r_2, t_2, w_2)$, with $r_1 \geq 1$ since $|E(\mathbb{Z}/p\mathbb{Z})|$ is even. By Theorem \ref{strong computation}, one can see that this percentage is at most $\frac{9}{11}$. 
\end{enumerate}
\vspace{-0.2in}
 \end{proof}

This proof of the following theorem follows directly along the lines of the proof of Theorem \ref{thm:propoverallcurves}.
\begin{theorem}
Let $N=p_1^{\alpha_1}...p_k^{\alpha_k}$ with $2 \nmid N$ and $p,q \mid N$. The probability that a random strongly non-zero point $P \cong \left(P_1, ... , P_k\right)$ has $\nu _2\left(\order\left(P_i\right)\right)$ all equal for a random elliptic curve $E\left(\mathbb{Z}/N\mathbb{Z}\right) \cong E\left(\mathbb{Z}/p_1^{\alpha_1}\mathbb{Z}\right) \oplus ... \oplus E\left(\mathbb{Z}/p_k^{\alpha_k}\mathbb{Z}\right)$ is at most $\frac{78pq -5p -5q +12}{120pq}$.
\end{theorem}

\begin{corollary}
Let $N=p_1^{\alpha_1}...p_k^{\alpha_k}$ with $2 \nmid N$ and $p$ and $q$ the largest primes dividing $N$. The probability that $N$ is a strong S-pseudoprime for a random strongly non-zero point $P \cong \left(P_1, ... , P_k\right)$ on a random curve $E\left(\mathbb{Z}/N\mathbb{Z}\right) \cong E\left(\mathbb{Z}/p_1^{\alpha_1}\mathbb{Z}\right) \oplus ... \oplus E\left(\mathbb{Z}/p_k^{\alpha_k}\mathbb{Z}\right)$ is at most $\frac{78pq -5p -5q +12}{120pq}$.
\end{corollary}


\begin{thebibliography}{1}

\bibitem{Abatzoglou}
A. Abatzoglou, A. Silverberg, A. Sutherland and A. Wong, \emph{Deterministic elliptic curve primality proving for a special sequence of numbers}, {\bf The Open Book Series}, Vol. 1:1 (2013), 1--20.

\bibitem{AGP}
W.R. Alford, A. Granville and C. Pomerance, {\em There are infinitely many Carmichael numbers}, {\bf Annals of Mathematics}, Vol. 140 (1994), 703--722.

\bibitem{AKS} M. Agrawal, N. Kayal, N. Saxena, \emph{Primes is in $P$},  {\bf Annals of Mathematics}, Vol. 160 (2004), 781--793. 

\bibitem{REU2017}
L. Babinkostova, A. Hernandez-Espiet and H.J. Kim, {\em On Types of Elliptic Pseudoprimes} (arXiv:1710.05264).

\bibitem{B}
 D.J. Bernstion, {\emph Proving primality in essentially quartic random time}, {\bf Mathematics of Computation}, Vol. 76 (2007), 389-403. 

\bibitem{gcdsum}
K. A. Broughan, {\em The gcd-sum function}, {\bf Journal of Integer Sequences}, Vol 4 (2001)


\bibitem{Carella}
N. A. Carella, \emph{Sum of Divisors Function Inequality}, (arXiv:0912.1866).

\bibitem{DiamondShurman}
F. Diamond and J.Shurman, \emph{A First Course in Modular Forms}, {\bf Graduate Texts in Mathematics}, Vol.~228 , Springer-Verlag New York, 1st~ed., (2005).


\bibitem{EPT}
A. Ekstrom, C. Pomerance and D. S. Thakur, \emph{Infinitude  of  elliptic Carmichael numbers}, {\bf J. Aust. Math. Soc.} Vol. 92 (2012), 45--60.

\bibitem{Ekstrom}
A. Ekstrom. \emph{On the infinitude of elliptic Carmichael numbers}. Ph. D. thesis, University of Arizona (1999).

\bibitem{GK}
S. Goldwasser and J. Kilian, \emph{Almost All Primes Can Be Quickly Certified}, {\bf Proc. 18th Annual ACM Symposium on Theory of Computing}, (1986), 316--329.

\bibitem{Gordon-number}
D. M. Gordon, \emph{On the number of elliptic pseudoprimes}, {\bf Mathematics of Computations} Vol. 52:185 (1989), 231--245.

\bibitem{Gordon-psp}
D. M. Gordon, \emph{Pseudoprimes on elliptic curves}, {\bf Th{\'e}orie des nombres: Proceedings of the 1987 International Number Theory Conference}, deGruyter, Berlin, (1989), 290--305.


\bibitem{Tenenbaum}
A. Hildebrand and G. Tenenbaum, \emph{Integers without large prime factors}, {\bf Journal de Th\'eorie des Nombres de Bordeaux}, Vol. 5  (1993), 411--484.


\bibitem{Jakimczuk}
R. Jakimczuk, \emph{Two Topics in Number Theory: Sum of Divisors of the Primorial and Sum of Squarefree Parts}, {\bf International Mathematical Forum}, Vol. 12:7 (2017),  331--338.


\bibitem{K1}
N. Koblitz, \emph{Primality of the number of points on an elliptic curve over a finite field}, {\bf Pacific Journal of Mathematics}, 131:1 (1988), 157--165.

\bibitem{KC}
A. R. Korselt, \emph{Probl\`{e}me chinois}, {\bf L'interm\'{e}diare des math\'{e}maticiens}, Vol. 6 (1899), 142--143.

\bibitem{Lenstra-factoring}
H.W. Lenstra Jr., \emph{Factoring Integers with Elliptic Curves}, {\bf The Annals of Mathematics: Second Series}, Vol. 126:3 (1987), 649--673.

\bibitem{Lenstra-algs}
H. W. Lenstra, Jr., \emph{Elliptic curves and number-theoretic algorithms}, {\bf Proceedings of the International Congress of Mathematicians, Berkeley, California} (1986), 99--120.

\bibitem{Mazur}
B. Mazur, \emph{Rational Points of Abelian Varieties with Values in Towers of Number Fields}, {\bf Invent. Mathematics} Vol. 18 (1972), 183--266.


\bibitem{Muller}
S. M\"uller, \emph{On the existence and non-existence of elliptic pseudoprimes}, {\bf Mathematics of Computation}, Vol. 79 (2009), 1171--1190.

\bibitem{Murty}
U.S.R. Murty, \emph{Problems in Analytic Number Theory}, {\bf Readings in Mathematics},  Vol.~206,  Springer-Verlag, New York, 1st~ed., (2001).


\bibitem{Nicolas}
J.L. Nicolas, \newblock {On Highly Composite Numbers},
\newblock \emph {\bf Ramanujan Revisited: Proceedings of the Centenary conference, University of Illinois at Urbana-Champaign, June 1-5, 1987} (1988), 215--244.

\bibitem{Qin}
H. Qin, \emph{Anomalous primes of the elliptic curve $E_D:y^2=x^3+D$}, {\bf Proceedings of London Mathematics Society}, Vol. 3:112 (2016), 415--453.

\bibitem{P}
C. Pomerance, \emph{Primality testing: Variations on a theme of Lucas}, {\bf Congressus Numerantium}, Vol. 201 (2010), 301--312.

\bibitem{LP}
 H. Lenstra and C. Pomerance, \emph{Primality testing with Gaussian periods}, preprint, (2016). 

\bibitem{Rabin}
M.O. Rabin, \emph{Probabilistic algorithm for testing primality}, {\bf Journal of Number Theory}, Vol. 12 (1980), 128--138.

\bibitem{Rosser}
J. B. Rosser and L.  Schoenfeld,  \emph{Approximate formulas for some functions of prime numbers}, {\bf Illinois J. Math.}, Vol. 6:1 (1962),  64--94.

\bibitem{Schlage}
J. Schlage-Puchta, \emph{The non-existence of universal Carmichael numbers}, {In: From Arithmetic to Zeta-Functions, Ed. J. Sander, J.Steuding, R. Steuding}, {\bf Springer International Publishing}, (2016).

\bibitem{Silverman-Carmichael}
J.H. Silverman, \emph{Elliptic Carmichael Numbers and Elliptic Korselt Criteria}, {\bf Acta Arithmetica} Vol. 155:3 (2012), 233--246.

\bibitem{Silverman}
J.H. Silverman, {\em The Arithmetic of Elliptic Curves}, Vol.~106 of {Graduate Texts in Mathematics}.
	\newblock {\bf Springer-Verlag New York}, 1st~ed., (1986).
 
\bibitem{Washington}
L.C. Washington, {\em Number Theory: Elliptic Curves and Cryptography}, Vol. 50,  {\bf Discrete Mathematics and Its Applications},
  \newblock Chapman \& Hall/CRC, 2nd~ed., (2008).
  
\end{thebibliography}
\end{document}